\documentclass[11pt]{amsart}

\begin{document}
\newtheorem{lemma}{Lemma}[section]
\newtheorem{prop}[lemma]{Proposition}
\newtheorem{cor}[lemma]{Corollary}
\newtheorem{thm}[lemma]{Theorem}
\newtheorem{mthm}[lemma]{Meta-Theorem}
\newtheorem{con}[lemma]{Conjecture}
\newtheorem{problem}{Problem}

\theoremstyle{definition}
\newtheorem{rem}[lemma]{Remark}
\newtheorem{rems}[lemma]{Remarks}
\newtheorem{defi}[lemma]{Definition}
\newtheorem{ex}[lemma]{Example}

\newcommand{\N}{\mathbb{N}}
\newcommand{\Q}{\mathbb{Q}}
\newcommand{\Z}{\mathbb{Z}}
\newcommand{\F}{\mathbb{F}}
\newcommand{\SL}{\mathrm{SL}}
\newcommand{\PSL}{\mathrm{PSL}}
\newcommand{\GL}{\mathrm{GL}}
\newcommand{\EL}{\mathrm{EL}}
\newcommand{\St}{\mathrm{St}}
\newcommand{\Mat}{\mathrm{Mat}}
\newcommand{\Sym}{\mathrm{Sym}}

\pagestyle{plain}
\title[New examples]{New examples of finitely presented groups with strong fixed point properties}
\author{Indira Chatterji and Martin Kassabov}
\thanks{Partially supported by NSF grants DMS No. 0644613 and DMS No.~0600244.
The second author is also partially supported by
AMS Centennial Fellowship}

\maketitle

\section*{Introduction}
{
\renewcommand{\thelemma}{\arabic{lemma}}
In the seventies, the seminal work of Margulis~\cite{Ma} on super-rigidity showed
the lack of linear actions of $\SL_n(\Z)$ on other spaces than the obvious ones.
In the eighties work of Zimmer~\cite{Zi}, Witte-Morris~\cite{WM} and Weinberger \cite{W} established the lack of $\SL_n(\Z)$-actions on some compact manifolds and very recent work of Farb~\cite{Farb} and
Bridson-Vogtmann~\cite{BV} shows the lack of actions of $\SL_n(\Z)$ on many contractible spaces of dimension less than $n$. Hence, the following is almost a folklore statement.
\begin{mthm}
The group $\SL_\infty(\Z)$ cannot act non-trivially on any ``reasonable'' space.
\end{mthm}
The aim of this note is to use this Meta-Theorem to give a natural and explicit example of a
finitely presented group that cannot act on any reasonable space
(the only problem being that it is not clear exactly what a ``reasonable'' space is).
More precisely we prove the following:%
\footnote{It has been pointed out to us by D. Osin that it is possible to construct groups with similar properties using
relatively hyperbolic groups and the Higman's embedding theorem, but the resulting groups
are far from explicit.}
\begin{thm}
\label{main}
There exists
an explicit finite presentation of a group $\Gamma$ containing $\SL_\infty(\Z)$ and normally generated by it.
\end{thm}
Since $\SL_{\infty}(\Z)$ cannot act on various classes of spaces we can deduce the following:
\begin{mthm}
The group $\Gamma$ is a finitely presented group which cannot act non trivially on
any ``reasonable'' space.
\end{mthm}
Notice that the lack of action of $\SL_{\infty}(\Z)$
on a various class of spaces makes full use of the torsion.
Recently, Fisher and Silberman in~\cite{FS} give a criteria
for a property (T) group to not admit any non-trivial
volume preserving smooth action on a compact manifold,
and their work apply to groups with little or no torsion.%
\footnote{It seems possible to combine the methods developed in~\cite{EJ-Z} with the one
in the paper to produce a group $\Gamma$ with property (T) and similar
fixed point properties.}
As a corollary of Theorem \ref{main}, we recover a particular case of a theorem of Arzhantseva et al.~\cite{crowd}.
\begin{cor}[Arzhantseva, Bridson, Januszkiewicz, Leary, Minasyan, Swiat\-kow\-ski]
\label{Aal}
There is a finitely presented group $\Gamma$ that cannot act
by homeomorphisms on a contractible manifold
of finite dimension, or cellularly on a uniformly locally finite CAT(0) cell complex,
by homeomorphisms on a finite dimensional $\Z_p$-acyclic manifold
over $\Z_p$ for $p=2,3$ or on a compact manifold.
\end{cor}
Using very different techniques in~\cite{crowd}, Arzhantseva et al.
actually prove the existence of finitely generated groups with much stronger fixed points properties and
that moreover can be chosen simple and with property (T).

That $\SL_n(\Z)$ (for $n\geq 3$) cannot act on a CAT(0) cell complex of dimension less than $n-1$
without a fixed point comes from the fact that $\SL_n(\Z)$
cannot act by semi-simple isometries on a CAT(0) space, see Farb's work on property FA$_n$~\cite{Farb}.
That $\SL_n(\Z)$ (for $n\geq 3$) cannot act non-trivially on a
contractible manifold of dimension less than $n$ is a deep result of
Bridson-Vogtmann~\cite{BV} and that $\SL_{\infty}(\Z)$ cannot act on a
compact manifold is a remark by Weinberger, see~\cite{WZ}.

The paper is organized as follows:
Section~\ref{rings} describes a finitely presented ring that contains matrices
of arbitrarily high dimension, and Section~\ref{magicfunctor} explains how to get
finitely presented groups out of those. Section~\ref{SLaction}
summarizes known results on the lack of action of $\SL_\infty({\Z})$ on various classes of spaces.
Section~\ref{G_0} explain how to get those groups to contain $\SL_\infty({\Z})$
and finally, Section~\ref{proofs} ties everything together
to build $\Gamma$ with the required properties.

{\em Acknowledgements:}
We thank Martin Bridson for Remark~\ref{tree},
Karen Vogtmann for explanations on~\cite{BV} and
David Fisher the many useful comments.
This paper started from conversation at the AIM workshop on percolation held in May 2008,
and we thank the Institute and the organizers for that opportunity.
}
\section{Crazy Rings}
\label{rings}
Let $S$ be a countable set. Let $\Mat_S(\Z)$ denote the ring of matrices where the rows and columns are indexed
by the elements in $S$ such that there are only finitely many non-zero entries in each row and column.
This ring naturally acts on the direct sum $\Z^{\oplus S}$ of copies of $\Z$ indexed by the set $S$:
every element of $\Z^{\oplus S}$ contains only finitely many non-zero entries.

We also define the two-sided ideal $\Mat_S^0(\Z)$ in $\Mat_S(\Z)$ consisting of all matrices
with finitely many non-zero entries. This is an ideal in $\Mat_S(\Z)$, because the matrices in $\Mat_S(\Z)$
contain only finitely many non-zero entries in each row and each column. Notice that
$\Mat_S^0(\Z)$ is not a subring of $\Mat_S(\Z)$ since it does not contain the unit element.
Let $E_{s,s'}$ denote the elementary matrix in $\Mat_S^0(\Z)$
which has $1$ in the $s$-th row and $s'$-th column and zeroes everywhere else.

Let $G$ be a group acting on the set $S$. There is a homomorphism
\begin{eqnarray*}
\phi_0: \Z[G] &\to& \Mat_S(\Z)\\
g&\mapsto&\sum_{s\in S}E_{s,gs}
\end{eqnarray*}
associated with this action, which sends each group element to the corresponding
permutation matrix.

The homomorphism $\phi_0$ defines left and right actions of the group algebra
$\Z[G]$ on $\Mat_S^0(\Z)$. This allows us to define the semi-direct product $R_{G,S}$ of $\Z[G]$ and $\Mat_S^0(\Z)$.
As an abelian group, the ring $R_{G,S}$ is just the direct sum $\Z[G]\oplus \Mat_S^0(\Z)$, but the multiplication is defined by
$$(f,m).(f',m)= (ff', \phi_0(f)m' + m\phi_0(f') + mm'),$$
where
$f,f'\in \Z[G]$ and $m,m' \in \Mat_S^0(\Z)$. In particular we have
$ g E_{s,s'} =  E_{gs,s'}$ and $E_{s,s'} g = E_{s,g^{-1}s'}$.
It is clear from the construction that $\Mat_S^0(\Z)$ is a two sided ideal in $R_{G,S}$.
The homomorphism $\phi_0$ can be extended,
using the natural embedding of $\Mat_S^0(\Z)$ into $\Mat_S(\Z)$,
to a homomorphism $\phi$ from $R_{G,S}$ to $\Mat_S(\Z)$.
Thus we have
$$
\Mat_S^0(\Z) \lhd \phi(R_{G,S}) \subset \Mat_S(\Z)
$$
which gives $\Z^{\oplus S}$ an $R_{G,S}$-module structure.
\begin{rem}
For many infinite groups $G$ and actions on sets $S$,
the homomorphisms $\phi_0$ and $\phi$ are injective
and we can define $R_{G,S}$ as the subring of $\Mat_S(\Z)$
generated by $\Mat_S^0(\Z)$ and the image of $\phi_0$.
However such definition does not give the same ring in general
(it is always false if the set $S$ is finite).
\end{rem}
\begin{defi}
For a group $G$ acting on a set $S$, we say that the action is \emph{almost transitive} if the $G$-action on $S$ has finitely many orbits. We say that the action is \emph{almost two-transitive} if the diagonal action of $G$ on $S\times S$ has finitely many orbits.\end{defi}
\begin{lemma}
\label{2tran}
(a) If a finitely generated group $G$ acts transitively (or almost transitively) on a set $S$, then the ring $R_{G,S}$ is finitely generated as an associative ring;

(b) If moreover $G$ is finitely presented, the action on $S$ is almost two-transitive,
and if the stabilizer in $G$ of a point in $S$ is finitely generated
then the ring $R_{G,S}$ is finitely presented as an associative ring.
\end{lemma}
\begin{proof}
(a) If the action is transitive, then the ring $R_{G,S}$ is generated by the generators of $G$ (and their inverses)
and $E_{s,s}$ for some fixed $s\in S$. Indeed, for any $g,g'\in G$,
$$
\phi_0(g) E_{s,s} \phi_0(g') =
\left(\sum_{t\in S}E_{t,gt}\right)E_{s,s}\left(\sum_{u\in S}E_{u,g'u}\right)
=E_{g^{-1}s, g's}.
$$
For an almost transitive action, one adds an extra generator for every pair of orbits of $G$ in $S$.

(b) For simplicity we will assume that the action of $G$ on $S$ is transitive.
Let $G=\langle X | W\rangle$ be a finite presentation of the group $G$.
Fix a point $s_0\in S$ and let $Y$ be a finite generating set for the stabilizer of $s_0$ in $G$.
Let $T\subseteq G$ be a finite set such that $\{(s_0,ts_0)\}_{t\in T}$ is a set of representatives
of the orbits of $G$ on $S\times S\setminus\Delta$, where $\Delta$ is the diagonal in $S \times S$.
Let $A$ be the ring generated by $X\cup X^{-1}\cup E$, where $E$ denotes the matrix $E_{s_0,s_0}$,
subject to the following relations:
\begin{eqnarray*}
W, & &\\
y E &=& E y = E \hbox{ for all }y\in Y,\\
E E &=& E,\\
E t E &=& 0\hbox{ for all }t \in T.
\end{eqnarray*}
Then the subring generated by $X\cup X^{-1}$ is isomorphic to $\Z[G]$.
The second relation gives a bijection between the set $G E G$ and the set of all
elementary matrices $E_{s',s''}$ where $s',s''\in S$.
Finally, the last two relations show that the elements $E_{s,s'}$ do
multiply as the elementary matrices in $\Mat_S^0(\Z)$. Therefore, the ring $A$ is isomorphic to $R_{G,S}$.
\end{proof}
In the main construction of the paper we need an example of an action which satisfies
the conditions of Lemma~\ref{2tran}(b). Such actions are hard to come by and all examples known to the authors are related to the Thompson's groups.
\begin{ex}
For any of the Thompson's groups $F$ or $T$
acting on the diadic points in the interval $(0,1)$, the action is almost two-transitive. Indeed, on can think of $F$ as the group of piecewise linear functions with the slopes being powers of 2 and diadic break points. It is then a classical result (see for instance~\cite{CFP}) that for diadic numbers $a,b,x,y\in(0,1)$ so that $a<b$ and $x<y$, there is an element $f\in F$ such that $f(a)=x$ and $f(b)=y$.
It is known that both $F$ and $T$ are finitely presented and that the stabilizer in $F$ (respectively $T$) of a dyadic  point in $(0,1)$ is isomorphic to $F \times F$ (respectively $F$).
\end{ex}
%
\section{The functors $\St_n$ and $\EL_n$}
\label{magicfunctor}
In this section we will describe two functors from the category of unital associative rings
to the category of groups.
For $R$ a unital associative ring, the group $\St_n(R)$, where $n\geq 3$, is generated by
$e_{ij}(r)$ for $1\leq i,j \leq n$ and $r\in R$ subject to the relations:
\begin{equation}\label{Stn}
\begin{array}{l}
\displaystyle
e_{ij}(r_1)\cdot e_{ij}(r_2) = e_{ij}(r_1+r_2)
\\
\displaystyle
\left[e_{ij}(r_1), e_{jk}(r_2)\right] = e_{ik}(r_1 r_2)
\\
\displaystyle
\left[e_{ij}(r_1), e_{pq}(r_2)\right] = 1 \quad \mbox{if } j\not = p \mbox{ and } i\not=q.
\end{array}
\end{equation}
We denote by $\EL_n(R)$ the image of $\St_n(R)$ in the multiplicative group of $n\times n$ matrices
with entries in $R$ under the obvious map sending
$e_{ij}(r)$ to the elementary matrix with $ij$-entry equal to $r$.

The following facts are (relatively) well known:
\begin{lemma}
\label{factsaboutSt}
(a) If $R$ is a finitely generated ring and $n\geq 3$ then both groups
$\St_n(R)$ and $\EL_n(R)$ are finitely generated;

(b) If $R$ is finitely presented and $n\geq 4$ then $\St_n(R)$ is finitely presented;

(c) If $n \geq 3$ then both $\St_n(R)$ and $\EL_n(R)$ are normally generated by $e_{12}(1)$.

(d) If $V$ is an $R$-module then $V^{\oplus n}$ is both an $\St_n(R)$-module and an $\EL_n(R)$-module;
\end{lemma}
\begin{proof}
Let $X$ be a finite generating set for the ring $R$, and assume that $X$ contains $1$.

(a) The group $\St_n(R)$ is generated by
$$
\{e_{ij}(x)\hbox{ such that }x\in X\}.
$$
Hence the group $\EL_n(R)$ is generated by the image of this generating set.

(b) In \cite{KMcC}, it is shown that if $\Z\langle T\rangle$ is the free associative ring
generated by a finite set $T$,
then $\St_n(\Z\langle T\rangle)$ is finitely presented. So, if $R=\Z\langle X\rangle/\langle Y \rangle$ is a finitely presented ring, for each relation $y\in Y$ in the presentation of $R$, there will be at most one relation to add to the presentation of $\St_n(\Z\langle T\rangle)$ to get the presentation of $\St_n(R)$.

(c) Any generator $e_{ij}(r)$ can be obtained form $e_{12}(1)$ using several commutators since $[e_{12}(1),e_{2j}(r)]=e_{1j}(r)$ and $[e_{i1}(1),e_{1j}(r)]=e_{ij}(r)$.

(d) That part follows from $V^{\oplus n}$ being naturally a $\Mat_n(R)$-module.
\end{proof}

The kernel of the projection from $\St_n(R)$ to $\EL_n(R)$ is denoted by $K_{2,n}(R)$.
There are several ways to construct elements in this kernel, as shown in the following:
\begin{lemma}
\label{k2}
(a) Let $E$ be a idempotent element in $R$, i.e., $E^2=E$. Then the product
$\tilde r(E)=\left( e_{12}(-E)e_{21}(E)e_{12}(-E)\right)^4$ is in $K_{2,n}(R)$, i.e.,
in the kernel of the projection from $\St_n(R)$ to $\EL_n(R)$;

(b) If $R=\Z$ then the group $K_{2,n}(\Z)$ is abelian of order $2$ and is generated by $\tilde r(1)$.
\end{lemma}
\begin{proof}
(a) A direct computation shows that the image of the product
$$
\left( e_{12}(-E)e_{21}(E)e_{12}(-E)\right)^2
$$
in $\EL_n(R)$ is a
diagonal matrix with $1-2E$ and $1$'s on the diagonal. Therefore its square is the identity, since
$(1-2E)^2=1$.

(b) The element $\tilde r(E)$ is equal to the Steinberg symbol $\{ 1-2E, 1-2E\}$, thus
$\tilde r(1)= \{-1,-1\}$, which is known to be the only non-trivial
element in $K_{2,n}(\Z)$, see Corollary 10.2 in~\cite{milnor}.
\end{proof}

\section{The groups $\Gamma_0$ and $\Delta_0$}
\label{G_0}
Let $G$ be a group acting on an infinite set $S$, such that the conditions of
Lemma~\ref{2tran} (b) are satisfied.
We will be looking at a quotient of
the Steinberg group $\St_4(R_{G,S})$ over the ring $R_{G,S}$, discussed in Section~\ref{magicfunctor}.

Since the ring $R_{G,S}$ acts on $\Z^{\oplus S}$ there is a natural action of
$\St_4(R_{G,S})$ and of $\EL_4(R_{G,S})$ on
$$
V = \Z^{\oplus S} \oplus \Z^{\oplus S} \oplus \Z^{\oplus S} \oplus \Z^{\oplus S}.
$$
Let $E$ be the generator of Lemma~\ref{2tran} (b), which satisfies $E^2=E$.
Denote by $\tilde{r}(E)$ the product $\left( e_{12}(-E)e_{21}(E)e_{12}(-E)\right)^4$.
\begin{defi}
Define
$$
\Gamma_0=\St_4(R_{G,S})/\langle \tilde r(E) \rangle.
$$
\end{defi}
\begin{lemma}
\label{G0}
The quotient group $\Gamma_0$ is finitely presented and acts naturally on $V$, i.e.,
$\tilde r(E)$ is in the kernel of the map $\St_4(R_{G,S}) \to \EL_4(R_{G,S})$.
\end{lemma}
\begin{proof}
By Lemma~\ref{factsaboutSt} (b), the group  $\St_4(R_{G,S})$ is finitely presented.
Adding the one extra relation $\tilde r (E)=1$ to this presentation gives
a presentation of $\Gamma_0$. By Lemma~\ref{k2} this relation is in the kernel of the projection.
\end{proof}
\begin{defi}
Let $\Delta_0$ be the subgroup of $\Gamma_0$ generated by
$$
\Delta_0=\left\langle e_{i,j}(x)\hbox{ such that }x \in I_{G,S}\right\rangle,
$$
where $I_{G,S}=\langle E \rangle \subseteq R_{G,S}$ is the standard copy of $\Mat^0_S(\mathbb Z)$ in $R_{G,S}$,
which is the same as the ideal corresponding to $\Mat_S^0(\Z)$ in the description of $R_{G,S}$
as a semi-direct product.
\end{defi}
\begin{lemma}
The group $\Delta_0$ is isomorphic to $\SL_\infty(\Z)$.
\end{lemma}
\begin{proof}
Since the set $S$ is countable we can list its elements $\{s_1,s_2,\dots,s_n,\dots\}$.
The group $\Delta_0$ is generated by the elements $e_{ij}(E_{s_p,s_q})$ for $p,q \in \N$.
We want to show that these elements generate a group isomorphic to $\SL_\infty(\Z)$. More precisely, they correspond to generators $e_{(i,s_p),(j,s_q)}$ of the
Steinberg group where the rows/columns are indexed by the elements in the countable set
$\{1,2,3,4\} \times S$. It is enough to show that these elements satisfy all
relations of the infinite dimensional analog of the Steinberg group over $\Z$, which
can be verified directly (this follow from the isomorphism $\St_n(\Mat_k(R)) = \St_{nk}(R)$,
by taking a limit when $k \to \infty$).
The element $\tilde r(E)$ was chosen to have its image in $\St_\infty(\Z)$ conjugated to
$\tilde r(1)$, which is the generator of $K_{2,n}(\Z)$, therefore $\Delta_0$ is
isomorphic to $\SL_\infty(\Z)$.
\end{proof}
\begin{rem}
If we drop the assumption that the set $S$ is countable then the group $\Delta_0$ is not isomorphic to $\SL_\infty(\Z)$ (since it is not countable) but to $\SL_\aleph(\Z)$ for some cardinal $\aleph = |S|$.
However this group contains and is normally generated by $\SL_\infty(\Z)$, thus
all results from Section~\ref{SLaction} still hold.
\end{rem}
\begin{lemma}
\label{actionV}
Let $g= e_{12}(1)\in \Gamma_0$ and $h=e_{34}(E_{s,s})\in\Delta_0<\Gamma_0$, then

(a) The element $g$ normally generates $\Gamma_0$;

(b) The elements $g$ and $h$ commute;

(c) Both $g$ and $h$ act on $V$ with infinitely many fixed points,
infinitely many infinite orbits and no finite orbits apart from fixed points.
\end{lemma}
\begin{proof} According to Lemma~\ref{G0}, the projection $\St_4(R_{G,S}) \to \EL_4(R_{G,S})$ factors through $\Gamma_0$. According to Lemma \ref{factsaboutSt} (b), a pre-image of the element $g$ in $\St_4(R_{G,S})$ normally generates it, which shows part (a). Part (b) follows from the Steinberg relations (\ref{Stn}).
For part (c), a direct computation shows that $g^n= e_{12}(n)$ and that $h^n=e_{34}(nE_{s,s})$.
Thus the orbit of a point $p \in V$ under the action of the group generated by $g$ (respectively $h$)
is infinite if the second (respectively the fourth) coordinate is not zero
(respectively not annihilated by $E_{s,s}$).
Otherwise, $p$ is fixed by
this subgroup.
\end{proof}


\section{Actions of $\SL_\infty(\Z)$}
\label{SLaction}
The aim of this section is to show that $\SL_\infty(\Z)$ cannot act by homeomorphisms on a contractible manifold
of finite dimension, or cellularly on a uniformly locally finite CAT(0) cell complex,
by homeomorphisms on a finite dimensional $\Z_p$-acyclic manifold
over $\Z_p$ for $p=2,3$ or on a compact manifold. We need a few properties of $\SL_\infty(\Z)$ to give the proof.
Let $\GL_\infty(\Z)$ denote the group of invertible infinite matrices with entries in $\Z$
which differ from the identity matrix only at finitely many places.
Let $\SL_\infty(\Z)$ be the subgroup of index $2$ in $\GL_\infty(\Z)$
which consists of all matrices of determinant $1$.
Another way to define $\SL_\infty(\Z)$ is as the limit of the increasing sequence
$$
\SL_2(\Z) \subset \SL_3(\Z) \subset \dots \subset \SL_n(\Z) \subset \SL_{n+1}(\Z) \subset \dots
$$
Notice that the group $\SL_\infty(\Z)$ is not finitely generated but it is
recursively presented.%
\footnote{
Most of the results of this section apply also to the group $\St_\infty(\Z)$
which is a $2$ sheeted central extension of $\SL_\infty(\Z)$.
Using this group in the main construction is equivalent to removing
the relation $\tilde r$ in the definition of the group $\Gamma_0$.
}
We start with several easy lemmas:
\begin{lemma}
\label{normgenSLinfty}
Every elementary matrix with a 1 off the diagonal normally generates $\SL_\infty(\Z)$.
\end{lemma}
\begin{proof}
Using commutators with permutation matrices one can get any elementary matrix,
and those are known to generate $\SL_n(\Z)$.
We can do that for any $n$, hence for $\SL_{\infty}(\Z)$ as well.
\end{proof}
\begin{lemma}
\label{finiteaction}
For $k\in\N$, there is $n=n(k)$ big enough so that any action of $\SL_n(\Z)$
on a finite set with less than $k$ points is trivial.
\end{lemma}
\begin{proof}
If $\SL_n(\Z)$ acts on a set of $k$ points non-trivially,
it means that there is a non-trivial homomorphism $\SL_n(\Z)\to \Sym(k)$,
where $\Sym(k)$ is the permutation group of $k$ elements.
But the kernel of this homomorphism is a normal subgroup of index at most $k!$
(the cardinality of $\Sym(k)$). However, according to \cite{BLS}, normal subgroups in $\SL_n(\Z)$ contain congruence subgroups, hence the smallest finite quotient of
$\SL_n(\Z)$ is the simple group $\PSL_n(\Z/2\Z)$ and has size approximately $2^{n^2-1}$, which completes the proof.
\end{proof}
\begin{lemma}
\label{nofinitequots}
The group $\SL_\infty(\Z)$ has no finite quotients.
\end{lemma}
\begin{proof}
The smallest finite quotient of $\SL_n(\Z)$ is $\PSL_n(\Z/2\Z)$ and has size approximately $2^{n^2-1}$. Since the size goes to infinity as $n$ tends to infinity,
we deduce that $\SL_\infty(\Z)$ has no finite quotient.
\end{proof}
\bigskip
The following fixed point result is not sharp but is enough for our purpose.
\begin{thm}[Farb, Corollary 1.2 of~\cite{Farb}]\label{FarbS}
Any simplicial action of $\SL_n(\Z)$ ($n\geq 3$) on a CAT(0) cell complex of dimension
$n-2$ has a globally fixed point.
\end{thm}
The lemmas above can be combined with this fixed point property to show:
\begin{cor}
\label{noaction}
Any simplicial action of $\SL_\infty(\Z)$ on a uniformly locally finite CAT(0) cell complex is trivial.
\end{cor}
\begin{proof}
Let $X$ be a uniformly locally finite CAT(0) cell complex, where the degree of each vertex is less than $N$.
Pick an integer $n$ big enough, so that $n>\dim(X)$ and $2^{n^2-1} > N!$.
Let $G$ denote a copy of $\SL_n(\Z)$
sitting inside $\SL_\infty(\Z)$. By Theorem~\ref{FarbS} the action of $G$ on $X$ has a fixed point.
Let $F\subseteq X$ denote the set of fixed points in $X$ under the $G$-action.
For any fixed point $x_0 \in F$ the group $G$
acts on the link of $x_0$ which is a finite set with less than $N$ elements.
By Lemma~\ref{finiteaction}
this action is trivial, so that any neighbour of $x_0$ is also in $F$.
Therefore $F$ contains all vertices in $X$, thus
$G$ acts trivially on $X$. However by Lemma~\ref{normgenSLinfty} the group $G$ normally generates
$\SL_\infty(\Z)$ therefore the whole group $\SL_\infty(\Z)$ acts trivially on $X$.
\end{proof}
\begin{rem}\label{tree}
M. Bridson noticed that it is possible to construct
an nontrivial action (even without fixed points) of $\SL_\infty(\Z)$
on an infinite locally finite tree. This does not contradict his result in~\cite{Br},
because every copy of $\SL_n(\Z)$ has many fixed points.
The construction (presented below) exploits the fact that
$\SL_\infty(\Z)$ is not finitely generated:

First construct an increasing sequence of subgroups $G_i$ in $\SL_\infty(\Z)$ such that:
the indexes $|G_i: G_{i+1}|$ are finite and $\SL_\infty(\Z) = \cup G_i$.
The group $G_n$ can be chosen as the group of all matrices which, when viewed modulo 2,
are ``upper triangular'' except in the top-left corner, i.e.,
$a_{ij} = 0 \mod 2$ when $i < j$ unless both $i$ and $j$ satisfy $i < n$ and $j <n$.
Observe that the index of $G_n$ in $G_{n+1}$ is the same as the index of
one of the maximal parabolic subgroups inside $\SL_{n}(\F_2)$ and is equal to $2^n-1$
(but goes to infinity as $n \to \infty$).
It is clear that $\SL_\infty(\Z)$ is the union of all groups $G_i$.

Using such a sequence one can build the forest $T$ of cosets: the vertices are the cosets $gG_i$
for some index $i$ and two vertices $gG_i$ and $hG_j$ are connected if $|i-j|=1$ and one coset
contains the other one. This graph is connected, i.e., it is a tree,
because $G= \cup G_i$ and is locally finite since
all indexes  $|G_i: G_{i+1}|$ are finite.

The action of $G$ on the set of cosets $G/G_i$ extends to a fixed point free action of $G$ on $T$.
This example shows that
the condition of uniform local finiteness in Corollary~\ref{noaction} can not be removed.
\end{rem}

Corollary~\ref{noaction} has several different forms depending on the class of spaces.
The proofs of these variants go along one of the following lines:
a) the group $\SL_\infty(\Z)$ contains every finite group as a subgroup, but ``large'' finite
groups cannot act on spaces of ``small'' complexity; b) any action of $\SL_n(\Z)$
on a space of ``complexity'' less than $n$ has a fixed point and one can use this fixed point
to show that $\SL_\infty(\Z)$ can not act on a space of ``complexity'' less than $n$.

Here are several analogs of Theorem~\ref{FarbS} and Corollary~\ref{noaction}:
\begin{thm}[Bridson-Vogtmann, Corollary 1.3 of~\cite{BV}]
\label{BVcontractible}
Any action of $\SL_n(\Z)$ by homeomorphisms on a contractible manifold of
dimension less than $n$ is trivial.
\end{thm}
This result can be easily extended to $\SL_\infty(\Z)$:
\begin{cor}
\label{noisoaction}
Any action of $\SL_\infty(\Z)$ by homeomorphisms on a contractible manifold is trivial.
\end{cor}
\begin{proof}
Let $X$ be a contractible manifold of finite dimension. Pick an integer $n$ so that $n>\dim(X)$.
Let $G$ denote a copy of $\SL_n(\Z)$ sitting inside $\SL_\infty(\Z)$.
According to Theorem~\ref{BVcontractible}, this action is trivial,
and therefore since by Lemma~\ref{normgenSLinfty},
the group $G$ normally generates $\SL_\infty(\Z)$,
we conclude that the whole group $\SL_\infty(\Z)$
acts trivially on $X$.
\end{proof}
\begin{thm}[Bridson-Vogtmann, Corollary 1.3 and Theorem 1.6 of~\cite{BV}]
\label{BVmodp}
If $n>3$ and $d<n$, then $\SL_n(\Z)$ cannot act non-trivially by homeomorphisms on
a $d$-dimensional $\Z_2$-acyclic homology manifold over $\Z_2$.
If $n>3$ is even and $d<n$, then any action of $\SL_n(\Z)$
by homeomorphisms on a $d$-dimensional $\Z_3$-acyclic homology manifold over $\Z_3$ is trivial.
\end{thm}
Again, one can extend this result to $\SL_\infty(\Z)$:
\begin{cor}
\label{nomodpaction}
For $p=2,3$, any action of $\SL_\infty(\Z)$ by homeomorphisms on a finite dimensional
$\Z_p$-acyclic manifold over $\Z_p$ is trivial.
\end{cor}
\begin{proof}
Let $X$ be a finite dimensional $\Z_p$-acyclic manifold over $\Z_p$.
Pick an even integer $n$ so that $n>\dim(X)$.
Let $G$ denote the copy of $\SL_n(\Z)$ sitting inside $\SL_\infty(\Z)$.
According to Theorem~\ref{BVmodp}, this action is trivial,
and therefore since by Lemma~\ref{normgenSLinfty},
the group $G$ normally generates $\SL_\infty(\Z)$,
we conclude that the whole group $\SL_\infty(\Z)$
acts trivially on $X$.
\end{proof}
In the same vein, Weinberger noticed the following.
\begin{prop}[Weinberger, see~\cite{WZ} Proposition 1]
\label{WZ}
The infinite special linear group $\SL_\infty(\Z)$ does not act topologically,
non-trivially, on any compact manifold, or indeed on any manifold whose homology with
coefficients in a field of positive characteristic is finitely generated.
\end{prop}
All these results suggest the meta-theorem quoted in the Introduction.
\section{The group $\Gamma$ and proof of the main results}
\label{proofs}
Let $\Gamma_0$ be the group defined in Section~\ref{G_0} and let $g$ and $h$ be the elements of $\Gamma_0$
defined in Lemma~\ref{actionV}. Notice that the group $\Gamma_0$ maps onto $\SL_4(\Z)$, by composition with the map $\St_4(R_{G,S})\to\EL_4(\Z)=\SL_4(\Z)$ induced by the map $R_{G,S}\to\Z$, hence it does act on a CAT(0) space, so is not good enough for our purpose.
Let $\bar\Gamma$ denote the free product of two copies of $\Gamma_0$, and let $\bar\rho_i$ be
the two embeddings of $\Gamma_0$ in $\bar\Gamma$. Finally let us define
$$
\Gamma=\Gamma_0*\Gamma_0/
\langle \bar \rho_1(g) = \bar \rho_2(h),\bar \rho_2(g) = \bar \rho_1(h) \rangle.
$$
That is, $\Gamma$ is the quotient of $\bar\Gamma$ by the relations
$$
\bar \rho_1(g) = \bar \rho_2(h)
\quad
\bar \rho_2(g) = \bar \rho_1(h).
$$
Another way to define $\Gamma$ is as the amalgamated product of two copies of $\Gamma_0$
over the subgroup $\Z\oplus \Z$ generated by $g$ and $h$, however the two subgroups
are identified via a nontrivial twist of $\Z \oplus \Z$. This shows that the group $\Gamma$
is non trivial and that the maps $\rho_1,\rho_2 :\Gamma_0 \to \Gamma$, induced from $\bar \rho_1$ and
 $\bar \rho_2$, are embeddings.
Another way to see that is to construct an explicit action of $\Gamma$ on
some infinite set (see Remark~\ref{action_on set}).%
\footnote{Similar way of gluing two groups was used in~\cite{KaN} in order to construct groups with
a controlled profinite completion.}
\begin{proof}[Proof of Theorem~\ref{main}]
First, the group $\Gamma$ is finitely presented since it is a quotient of $\bar\Gamma$
using two more relations only. By construction as amalgamated product, the group $\Gamma$ is infinite.
According to Lemma \ref{actionV}, the group $\Gamma_0$ is normally generated by $g$,
so that $\Gamma$ is normally generated by $\rho_1(g)$ and $\rho_2(g)$. So, by construction, $\Gamma$ is normally generated by the two copies of $\SL_\infty(\Z)$ isomorphic to $\rho_1(\Delta_0)$ and $\rho_2(\Delta_0)$. But in fact, the copy of $\SL_\infty(\Z)$ isomorphic to $\rho_1(\Delta_0)$ is enough. Indeed, since $\bar\rho_2(g)=\bar\rho_1(h)$ it implies that $\rho_2(g)\in\rho_1(\Delta_0)$.
But $g$ normally generates $\Gamma_0$, hence $\rho_2(g)$ normally generates $\rho_2(\Delta_0)$ as well.
\end{proof}

\begin{proof}[Proof of Corollary~\ref{Aal}]
Suppose that $\Gamma$ acts either simplicially on some uniformly locally finite CAT(0) cell complex  $X$,
or by homeomorphisms on a contractible manifold $X$ of finite dimension,
on a $\Z_p$-acyclic manifold over $\Z_p$ for $p=2,3$ or on a compact manifold.
Since $h\in \SL_\infty(\Z) \simeq \Delta_0$ and $\rho_i: \Delta_0 \hookrightarrow \Gamma$ are injective, by Corollaries~\ref{noaction}, \ref{noisoaction} and~\ref{nomodpaction} or
Proposition~\ref{WZ}, the action of $\rho_1(h)$ and $\rho_2(h)$ on $X$ is trivial.
The defining relations
$\rho_{3-i}(h) = \rho_{i}(g)$ give us that both $\rho_{1}(g)$
and $\rho_{2}(g)$ act trivially on $X$.
However, the group $\Gamma$ is normally generated by these two elements and
therefore the whole group acts trivially on $X$.
\end{proof}
If we had variants of Corollaries~\ref{noaction} or~\ref{noisoaction},
one could generalize Corollary~\ref{Aal} to other classes of
topological spaces like general CAT(0) or systolic spaces that are uniformly locally finite, $\Z_p$-acyclic manifolds for $p>3$, etc,
which would give more evidence for the meta-theorem of the Introduction.
We finish with a direct proof that the group $\Gamma$ is not trivial:
\begin{rem}
\label{action_on set}
The group $\Gamma$ acts non-trivially on the direct sum $V$ of $4$ copies of $\Z^{\oplus S}$:
By Lemma~\ref{actionV} the group $\Gamma_0$ acts on $V$ and
the actions of both $g$ and $h$ on this countable set has infinitely many fixed points and
infinitely many infinite orbits (and no finite orbits other than fixed points).
Therefore there exists an involution
$\sigma$ on $V$ which conjugates the action of $h$ to the one of $g$
($\sigma$ is just some permutation
of the countable set $V$, we do not require that $\sigma$ preserves the linear structure of $V$).
This permutation allows us to construct a non-trivial action of $\Gamma$ on $V$: $\rho_1(\Gamma_0)$
acts as in Lemma~\ref{actionV} and $\rho_2(\Gamma_0)$ acts by the twist $\sigma$.
Our choice of the permutation $\sigma$ ensures that the relations
$\rho_1(g)= \rho_2(h)$ and $\rho_2(g)=\rho_1(h)$ are satisfied.
\end{rem}

\end{document}